\documentclass[12pt,a4paper]{amsart}
\usepackage{amsfonts}
\usepackage{amsthm}
\usepackage{amsmath}
\usepackage{amscd}
\usepackage{amssymb}
\usepackage[latin2]{inputenc}
\usepackage{t1enc}
\usepackage[mathscr]{eucal}
\usepackage{indentfirst}
\usepackage{graphicx}
\usepackage{graphics}
\usepackage{pict2e}
\usepackage{epic}
\usepackage{esint}
\numberwithin{equation}{section}
\usepackage[margin=2.9cm]{geometry}
\usepackage{epstopdf}

\theoremstyle{plain}
\newtheorem{Th}{Theorem}[section]
\newtheorem{Lemma}[Th]{Lemma}
\newtheorem{Cor}[Th]{Corollary}

 \theoremstyle{definition}

\newtheorem{Rem}[Th]{Remark}
\newtheorem{?}[Th]{Problem}

\begin{document}

\title[Conformal bounds for the first eigenvalue]{Conformal bounds for the first eigenvalue of the $\left(p,q\right)$-Laplacian system}

\author[M. Habibi Vosta Kolaei \and S. Azami]{Mohammad Javad Habibi Vosta Kolaei \and Shahroud Azami}

\address{ Department of pure Mathematics,  Faculty of Science,
Imam Khomeini \linebreak International University,
Qazvin, Iran.  }

\email{MJ.Habibi@Edu.ikiu.ac.ir}
\email{azami@Sci.ikiu.ac.ir}



\begin{abstract}
Consider $\left(M,g\right)$ as an $m$-dimensional compact connected Riemannian manifold without boundary. In this paper, we investigate the first eigenvalue $\lambda_{1,p,q}$ of the $\left(p,q\right)$-Laplacian system on $M$. Also, in the case of $p,q >n$ we will show that for arbitrary large $\lambda_{1,p,q}$ there exists a Riemannian metric of volume one conformal to the standard metric of $\mathbb{S}^{m}$.
\end{abstract}

\maketitle
\section{Introduction}
Finding bounds of the eigenvalue for the Laplacian on a given manifold is a key aspect in Riemannian geometry. Aa an example, studying eigenvalues that appear as solutions of the Dirichlet or Neumann boundary value problems for curvature functions, is interesting topic in geometric analysis. In recent years, because of the theory of self-adjoint operators, the spectral properties of linear Laplacian studied extensively. As an important example, mathematicians generally are interested in the spectrum of the Laplacian on compact manifolds with or without boundary or noncompact complete manifolds due to in these two cases the linear Laplacians can be uniquely extended to self-adjoint operators (see \cite{ga, ga1}).\\
Since the study of the properties of spectrum of Laplacian (specially in Dirichlet \linebreak condition) in infinitely stretched regions has applications in elasticity, \linebreak electromagnetism and quantum physics, it attracts attention of many \linebreak mathematicians and physicists. Recently Mao has proved the existence of discrete spectrum of linear Laplacian on a class of $4$-dimentional rotationally symmetric \linebreak quantum layers, which are noncompact noncomplete manifolds in \cite{ma1}. \\
Consider $M$ as a compact, complete, simply connected Riemannian \linebreak manifold. Let $u: M\longrightarrow \mathbb{R}$ be a smooth function on $M$ or $u \in W^{1,p}\left(M\right)$ where $W^{1,p}\left(M\right)$ is the Sobolev space. The $p$-Laplacian of $u$ for $1<p<\infty$ is defined as
\begin{align}\label{yek}
 \Delta_{p} u &= div( |\nabla u|^{p-2} \nabla u )\\
  &= | \nabla u |^{ p-2} \Delta u + (p-2)|\nabla u|^{p-4} ( Hess \,u)( \nabla u, \nabla u ), \nonumber
\end{align}
where
\begin{align*}
\left(Hess \, u\right)\left(X,Y\right) &= \nabla \left(\nabla u\right)\left(X,Y\right)\\
&= X.\left(Y.u\right) - \left(\nabla_{X}Y\right).u \quad X,Y \in \chi\left(M\right).
\end{align*}
Although the regularity theory of the $p$-Laplacian is very different from the usual Laplacian, many of the estimates for the first eigenvalue of the Laplacian \linebreak (for example for $p=2$) can be generalized to general $p$. As an important \linebreak example in \cite{ma}, you can find remarkable results in a case of closed manifolds with bounded Ricci curvature from below by $\left(m-1\right)K$ where $K>0$. The special case $K=0$ and general case $ K \in \mathbb{R}$ are studied in \cite{val} and \cite{na}, respectively.\\
Consider $g$ as a Riemannian metric on $M$. The conformal class of $g$ defined as
\begin{align*}
[g] = \lbrace fg | f \in C^{\infty}\left(M\right), f>0 \rbrace,
\end{align*}
also
\begin{align*}
G\left(n\right) = \lbrace \gamma \in Diff\left(\mathbb{S}^{n}\right) | \gamma^{*}can \in [can] \rbrace,
\end{align*}
for arbitrary natural $n$, denote the group of conformal diffeomorphisms of $\left(\mathbb{S}^{n}, can\right)$. It was proved before, for $n$ big enough the set
\begin{align*}
I_{n}\left(M,[g]\right) = \lbrace \phi : M \rightarrow \mathbb{S}^{n} | \phi^{*}can \in [g] \rbrace,
\end{align*}
of conformal immersions from $\left(M,g\right)$ to $\left(\mathbb{S}^{n}, can\right)$ is nonempty.
The spectrum of eigenvalues of geometric operators were studied before. As an example, for $m$-dimensional closed connected Riemannian manifold $M$ with metric $g$
\begin{align*}
Spec\left(g\right) = \lbrace 0 = \lambda_{0}\left(g\right) < \lambda_{1}\left(g\right) \leq \lambda_{2}\left(g\right) \leq ... \leq \lambda_{k}\left(g\right) \leq ... \rbrace,
\end{align*}
where $\lambda_{k}\left(g\right)$ denotes the $k$-th eigenvalue of Laplace operator. Furthermore,
\begin{align*}
\lambda_{k}^{c}\left(M, [g]\right) = \sup_{\tilde{g} \in [g]} \lambda_{k}\left(\tilde{g}\right) = \sup \lbrace \lambda_{k}\left(\tilde{g}\right)V\left(\tilde{g}\right)^{\frac{2}{m}}\rbrace,
\end{align*}
where $\tilde{g}$ is the metric conformal to $g$ and $V\left(\tilde{g}\right)$ is the volume element associated to $\tilde{g}$.\\
The conformal bound for the first eigenvalue of $p$-Laplacian system (\ref{yek}) was studied before in \cite{ma2}.
\begin{Th}[{\bf Matei \cite{ma2}}]
Let $M$ be an $m$-dimensional compact manifold and $1 < p \leq m$. If $g$ denotes the Riemannian metric on $M$ and $ n \in \mathbb{N}$ then
\begin{align*}
\lambda_{1,p}\left(M\right) \leq m^{\frac{p}{2}}\left(n+1\right)^{|\frac{p}{2}-1|}V_{n}^{c}\left(M,[g]\right)^{\frac{p}{m}},
\end{align*}
where
\begin{align*}
V_{n}^{c}\left(M,[g]\right) = \inf_{\phi \in I_{n}\left(M,[g]\right)} \sup_{\gamma \in G\left(n\right)} {\rm vol}\left(M, \left(\gamma \circ \phi \right)^{*}can \right).
\end{align*}
\end{Th}
\section{The $\left(p,q\right)$-Laplacian system}
The $\left(p,q\right)$-elliptic quasilinear system is defined as
\begin{equation}\label{se}
\left\{
\begin{array}{lr}
-\Delta_{p} u =+\lambda |u|^\alpha |v|^\beta v   & \text{in M }, \\
-\Delta_{q} v =+ \lambda |u|^\alpha |v|^\beta u  & \text{in M }, \\
 u=v=0 \quad \text{(Dirichlet)}\quad \text{or} \quad \nabla_{\delta} u= \nabla_{\delta} v =0 \quad \text{(Neumann)} &  \text{on}\  \partial \text{M},
\end{array} \right.
\end{equation}
where $\delta$ is the outward normal on $\partial M$, $ p>1$ , $ q>1$ and $ \alpha , \beta $ are real numbers such that
\begin{align*}
 \alpha >0 ,\, \beta>0 , \qquad \frac{\alpha +1}{p} + \frac{\beta +1}{q} = 1.
 \end{align*}
In this situation $\lambda$ is called an eigenvalue of system (\ref{se}) and $(u,v)$ are eigenfunctions corresponding to $\lambda$.\\
In the term of the first nontrivial eigenvalue of  the $(p,q)$-elliptic quasilinear system (\ref{se}), the first Neumann eigenvalue is defined as
\begin{align*}
\mu_{1,p,q} = inf \Big\lbrace \frac{1}{\mathcal{A}}\left(\frac{\alpha +1}{p} \int_{M} |\nabla u|^{p}  + \frac{\beta +1}{q} \int_{M} |\nabla v|^{q}\right) ; \quad u,v \in W^{1,p}\left(M\right) \setminus \lbrace 0 \rbrace, \mathcal{B}=\mathcal{E} =0 \Big\rbrace,
\end{align*}
where
\begin{align*}
\mathcal{A} =\int_{M}|u|^{\alpha +1}|v|^{\beta +1},
\end{align*}
and also
\begin{align*}
\mathcal{B}= \int_{M} \left(u|u|^{p-2} + |u|^{\alpha}|v|^{\beta}v\right),
\end{align*}
\begin{align*}
\mathcal{E}=\int_{M} \left(v|v|^{q-2} + |u|^{\alpha}|v|^{\beta}u\right).
\end{align*}
N. Zographopoulos in \cite{zog} has discussed the existence and uniqueness of the \linebreak solution of the $(p,q)$-elliptic quasilinear system (\ref{se}). This type of systems have been found in different cases in physics. For example to the study of transport of electron \linebreak temperature in a confined plasma and also to the study of electromagnetic \linebreak phenomena in nonhomogeneous super conductors, you can see \cite{brown, dan}. Also for more details in electrochemistry and nuclear reaction, you can find useful results in \cite{choi} or \cite{con}, respectively.\\
Let $\left(M,g\right)$ be an $m$-dimensional compact Riemannian manifold. The first Dirichlet eigenvalue of the system (\ref{se}) is defined as
\begin{align*}
\lambda_{1,p,q}\left(M\right) = \inf_{u,v \neq 0} \Big\lbrace \frac{1}{\int_{M}|u|^{\alpha +1}|v|^{\beta +1}dv} &\Big[\frac{\alpha +1}{p}\int_{M}|\nabla u|^{p}dv + \frac{\beta +1}{q}\int_{M}|\nabla v|^{q}dv \Big] \Big \rbrace,
\end{align*}
where
\begin{align*}
\left(u,v\right) \in W_{0}^{1,p}\left(M\right) \times W_{0}^{1,q}\left(M\right) \setminus \lbrace 0 \rbrace.
\end{align*}
As an example the second author has studied the first eigenvalue of the system (\ref{se}) in \cite{az}.
In this paper by inproving methods from Matei \cite{ma2}, we are going to study the first Dirichlet eigenvalue of the system (\ref{se}).\\
\section{The first case, $p,q \leq m$}
In this section we will prove that
\begin{Th}\label{yet}
Consider $M$ as an $m$-dimensional compact Riemannian manifold and also $1 < p,q \leq m$. If $\lambda_{1,p,q}$ denotes the first eigenvalue of the $\left(p,q\right)$-Laplacian system (\ref{se}) and $p \geq q$, then for arbitrary natural $n$ we get
\begin{itemize}
\item
If $p,q \geq 2$, then
\begin{align*}
\lambda_{1,p,q}\left(M\right) \leq \left(n+1\right)^{\frac{1}{2}p^{2}}m^{\frac{p}{2}}\left(V_{n}^{c}\left(M,[g]\right)\right)^{\frac{p}{m}}.
\end{align*}
\item
If $1< q,p < 2$, then
\begin{align*}
\lambda_{1,p,q}\left(M\right) \leq \left(n+1\right)^{-\frac{1}{2}q\left(q+1\right)} m^{\frac{q}{2}}\left(V_{n}^{c}\left(M,[g]\right)\right)^{\frac{q}{m}}.
\end{align*}
\item
If $1< q< 2\leq p$ then
\begin{align*}
\lambda_{1,p,q}\left(M\right) \leq \left(n+1\right)^{\frac{1}{2}\left(p^{2}-q\right)}m^{\frac{q}{2}}\left(V_{n}^{c}\left(M,[g]\right)\right)^{\frac{q}{m}}.
\end{align*}
\end{itemize}
\end{Th}
Before giving proof for this theorem, first of all, we consider two following lemmas.
\begin{Lemma}[{\bf Chebyshev's inequality \cite{be}}]
Consider $\lbrace a_{i} \rbrace_{i=1}^{n}$ and $\lbrace b_{i} \rbrace_{i=1}^{n}$ as two decreasing real sequences, then
\begin{align*}
\frac{1}{n}\sum_{i=1}^{n} a_{i}b_{i} \geq \left(\frac{1}{n} \sum_{i=1}^{n} a_{i}\right)\left(\frac{1}{n} \sum_{i=1}^{n} b_{i}\right).
\end{align*}
\end{Lemma}
\begin{Lemma}\label{lemyek}
Let $\phi : \left(M,g\right) \rightarrow \left(\mathbb{S}^{n}, can\right)$ be a smooth map which its level sets are zero measure, then there exist $\gamma, \delta \in G\left(n\right)$ and $p \geq q$ such that
\begin{itemize}
\item
If $p,q \geq 2$, then
\begin{align*}
\lambda_{1,p,q}\left(M\right) \leq \left(n+1\right)^{\frac{1}{2}p^{2}}\left(\frac{\alpha +1}{p}\int_{M}|d\psi|^{p} dv + \frac{\beta +1}{q}\int_{M}|d\eta|^{q} dv\right).
\end{align*}
\item
If $1<p,q<2$, then
\begin{align*}
\lambda_{1,p,q}\left(M\right) \leq \left(n+1\right)^{-\frac{1}{2}q\left(q+1\right)}\left(\frac{\alpha +1}{p}\int_{M}|d\psi|^{p} dv + \frac{\beta +1}{q}\int_{M} |d\eta|^{q} dv \right).
\end{align*}
\item
If $1<q<2\leq p$, then
\begin{align*}
\lambda_{1,p,q}\left(M\right) \leq \left(n+1\right)^{\frac{1}{2}\left(p^{2}-q\right)}\left( \frac{\alpha +1}{p}\int_{M} |d\psi |^{q} dv + \frac{\beta +1}{q} \int_{M}|d\eta |^{q} dv \right),
\end{align*}
\end{itemize}
where $\eta = \delta \circ \phi$ and $\psi = \gamma \circ \phi$.
\end{Lemma}
\begin{proof}
For $\alpha, \beta >0$ there exist $\gamma, \beta \in G\left(n\right)$ and for $\psi_{i} = \left(\gamma \circ \phi\right)_{i}$ and $\eta_{i} = \left(\delta \circ \phi\right)_{i}$ where $1\leq i \leq n+1$ we see
\begin{align*}
\lambda_{1,p,q}\left(M\right) \leq \frac{1}{\int_{M}|\check{\psi}_{i}|^{\alpha +1}|\check{\eta}_{i}|^{\beta +1}dv}\Big[\frac{\alpha +1}{p}\int_{M}|d\check{\psi}_{i}|^{p} dv + \frac{\beta +1}{q} \int_{M}|d\check{\eta}_{i}|^{q} \Big],
\end{align*}
where $\check{\eta}_{i}$ and $\check{\psi}_{i}$ are the decreasing rearrangement of $\eta_{i}$ and $\psi_{i}$ respectively, then
\begin{align*}
\lambda_{1,p,q}\left(M\right) \int_{M} |\check{\psi}_{i}|^{\alpha +1}|\check{\eta}_{i}|^{\beta +1} dv \leq \frac{\alpha +1}{p}\int_{M} |d\check{\psi}_{i}|^{p} + \frac{\beta +1}{q}\int_{M}|d\check{\eta}_{i}|^{q} dv.
\end{align*}
By taking summation from $i=1$ to $i = n+1$ from both sides we conclude that
\begin{align*}
\sum_{i=1}^{n+1}\left( \lambda_{1,p,q}\left(M\right) \int_{M} |\check{\psi}_{i}|^{\alpha +1}|\check{\eta}_{i}|^{\beta +1} dv \right) \leq \sum_{i=1}^{n+1}\left( \frac{\alpha +1}{p}\int_{M} |d\check{\psi}_{i}|^{p} + \frac{\beta +1}{q}\int_{M}|d\check{\eta}_{i}|^{q} \right),
\end{align*}
or
\begin{align*}
\lambda_{1,p,q}\left(M\right) \leq \frac{1}{\int_{M}\sum_{i =1}^{n+1} |\check{\psi}_{i}|^{\alpha +1}|\check{\eta}_{i}|^{\beta +1}dv}\Big[\frac{\alpha +1}{p}\int_{M}\sum_{i=1}^{n+1} |d\check{\psi}_{i}|^{p} dv + \frac{\beta +1}{q} \int_{M}\sum_{i=1}^{n+1} |d\check{\eta}_{i}|^{q} \Big].
\end{align*}
First of all, let $p,q \geq 2$, then
\begin{align*}
\sum_{i=1}^{n+1} |d\check{\psi}_{i}|^{p} = \sum_{i=1}^{n+1}\left(|d\check{\psi}_{i}|^{2}\right)^{\frac{p}{2}} &\leq \left(\sum_{i=1}^{n+1}|d\check{\psi}_{i}|^{2}\right)^{\frac{p}{2}} = |d\psi |^{p},
\end{align*}
and also
\begin{align*}
\sum_{i=1}^{n+1} |d\check{\eta}_{i}|^{q} = \sum_{i=1}^{n+1}\left(|d\check{\eta}_{i}|^{2}\right)^{\frac{q}{2}} &\leq \left(\sum_{i=1}^{n+1}|d\check{\eta}_{i}|^{2}\right)^{\frac{q}{2}} = |d\eta |^{q}.
\end{align*}
Now by Chebyshev's inequality we get
\begin{align}\label{char}
\sum_{i=1}^{n+1}|\check{\psi}_{i}|^{\alpha +1}|\check{\eta}_{i}|^{\beta +1} &\geq \frac{1}{n+1} \sum_{i=1}^{n+1}|\check{\psi}_{i}|^{\alpha +1}|\check{\eta}_{i}|^{\beta +1} \nonumber \\
&\geq \left(\frac{1}{n+1}\sum_{i=1}^{n+1}|\check{\psi}_{i}|^{\alpha +1}\right) \left(\frac{1}{n+1}\sum_{i=1}^{n+1}|\check{\eta}_{i}|^{\beta +1}\right) \nonumber \\
&\geq \left(\frac{1}{n+1}\sum_{i=1}^{n+1}|\check{\psi}_{i}|^{p\left(\alpha +1\right)}\right) \left(\frac{1}{n+1}\sum_{i=1}^{n+1}|\check{\eta}_{i}|^{q\left(\beta +1\right)}\right) \nonumber \\
&= \left(\sum_{i=1}^{n+1} \frac{1}{n+1}\left(|\check{\psi}_{i}|^{2}\right)^{\frac{p}{2}\left(\alpha +1\right)}\right) \left(\sum_{i=1}^{n+1} \frac{1}{n+1}\left(|\check{\eta}_{i}|^{2}\right)^{\frac{q}{2}\left(\beta +1\right)}\right).
\end{align}
By Jensen's inequality in (\ref{char}) we conclude that
\begin{align*}
\sum_{i=1}^{n+1}|\check{\psi}_{i}|^{\alpha +1}|\check{\eta}_{i}|^{\beta +1} &\geq \left(\sum_{i=1}^{n+1} \frac{1}{n+1}|\check{\psi}_{i}|^{2}\right)^{\frac{p}{2}\left(\alpha +1\right)} \left(\sum_{i=1}^{n+1} \frac{1}{n+1}|\check{\eta}_{i}|^{2}\right)^{\frac{q}{2}\left(\beta +1\right)} \\
&= \left(\frac{1}{n+1}\right)^{\frac{p}{2}\left(\alpha +1\right)}\left(\sum_{i=1}^{n+1}|\check{\psi}_{i}|^{2}\right)^{\frac{p}{2}\left(\alpha +1\right)} \left(\frac{1}{n+1}\right)^{\frac{q}{2}\left(\beta +1\right)}\left(\sum_{i=1}^{n+1}|\check{\eta}_{i}|^{2}\right)^{\frac{q}{2}\left(\beta +1\right)}\\
&= \left(n+1\right)^{-\frac{1}{2}\left(p\left(\alpha +1\right) + q\left(\beta +1\right)\right)}.
\end{align*}
What we have done is dependent on two essential issues, first $\sum_{i=1}^{n+1}|\psi_{i}|^{2} = \sum_{i=1}^{n+1}|\eta_{i}|^{2} =1$ and also  we know that the map $x \rightarrow x^{\frac{R}{2}}$ for $\frac{R}{2} \geq 1$ is concave. Now under consideration $p \geq q$ we have
\begin{align*}
\lambda_{1,p,q}\left(M\right) \leq \left(n+1\right)^{\frac{1}{2}p^{2}}\left(\frac{\alpha +1}{p}\int_{M}|d\psi|^{p} dv + \frac{\beta +1}{q}\int_{M}|d\eta|^{q} dv\right).
\end{align*}
In the case that $1< p,q <2$, since
\begin{align*}
 |\psi_{i}| \leq 1, \quad |\eta_{i}| \leq 1,
\end{align*}
and also the maps $x \rightarrow x^{\frac{p+1}{2}}$ and $x \rightarrow x^{\frac{q+1}{2}}$ are concave, by the similar way
\begin{align*}
\sum_{i=1}^{n+1}|\check{\psi}_{i}|^{\alpha +1}|\check{\eta}_{i}|^{\beta +1} &\geq \frac{1}{n+1}\sum_{i=1}^{n+1}|\check{\psi}_{i}|^{\alpha +1}|\check{\eta}_{i}|^{\beta +1}\\
&\geq \left(\frac{1}{n+1}\sum_{i=1}^{n+1}|\check{\psi}_{i}|^{\alpha +1}\right) \left(\frac{1}{n+1}\sum_{i=1}^{n+1}|\check{\eta}_{i}|^{\beta +1}\right) \nonumber \\
&\geq \left(\frac{1}{n+1}\sum_{i=1}^{n+1}|\check{\psi}_{i}|^{\left(p+1\right)\left(\alpha +1\right)}\right) \left(\frac{1}{n+1}\sum_{i=1}^{n+1}|\check{\eta}_{i}|^{\left(q+1\right)\left(\beta +1\right)}\right) \nonumber \\
&= \left(\sum_{i=1}^{n+1} \frac{1}{n+1}\left(|\check{\psi}_{i}|^{2}\right)^{\frac{p+1}{2}\left(\alpha +1\right)}\right) \left(\sum_{i=1}^{n+1} \frac{1}{n+1}\left(|\check{\eta}_{i}|^{2}\right)^{\frac{q+1}{2}\left(\beta +1\right)}\right).
\end{align*}
And also by Jensen's inequality, it concludes that
\begin{align*}
\sum_{i=1}^{n+1}|\check{\psi}_{i}|^{\alpha +1}|\check{\eta}_{i}|^{\beta +1} &\geq \left(\sum_{i=1}^{n+1} \frac{1}{n+1}|\check{\psi}_{i}|^{2}\right)^{\frac{p+1}{2}\left(\alpha +1\right)}  \left(\sum_{i=1}^{n+1} \frac{1}{n+1}|\check{\eta}_{i}|^{2}\right)^{\frac{q+1}{2}\left(\beta +1\right)} \\
&= \left(\frac{1}{n+1}\right)^{\frac{p+1}{2}\left(\alpha +1\right)}\left(\sum_{i=1}^{n+1}|\check{\psi}_{i}|^{2}\right)^{\frac{p+1}{2}\left(\alpha +1\right)}\\
&\times \left(\frac{1}{n+1}\right)^{\frac{q+1}{2}\left(\beta +1\right)}\left(\sum_{i=1}^{n+1}|\check{\eta}_{i}|^{2}\right)^{\frac{q+1}{2}\left(\beta +1\right)}\\
&= \left(n+1\right)^{-\frac{1}{2}\left(\left(p+1\right)\left(\alpha +1\right) + \left(q+1\right)\left(\beta +1\right)\right)}.
\end{align*}
Since $x\rightarrow x^{\frac{p}{2}}$ and $x \rightarrow x^{\frac{q}{2}}$ are convex, we see
\begin{align*}
\sum_{i=1}^{n+1}|d\check{\psi}_{i}|^{p} = \sum_{i=1}^{n+1}\left(|d\check{\psi}_{i}|^{2}\right)^{\frac{p}{2}} &\leq \left(n+1\right)^{1-\frac{p}{2}}\left(\sum_{i=1}^{n+1}|d\check{\psi}_{i}|^{2}\right)^{\frac{p}{2}} \\
&= \left(n+1\right)^{1-\frac{p}{2}}|d\psi |^{p},
\end{align*}
and
\begin{align*}
\sum_{i=1}^{n+1}|d\check{\eta}_{i}|^{q} = \sum_{i=1}^{n+1}\left(|d\check{\eta}_{i}|^{2}\right)^{\frac{q}{2}} &\leq \left(n+1\right)^{1-\frac{q}{2}}\left(\sum_{i=1}^{n+1}|d\check{\eta}_{i}|^{2}\right)^{\frac{q}{2}} \\
&= \left(n+1\right)^{1-\frac{q}{2}}|d\eta |^{q}.
\end{align*}
These together with $p \geq q$ conclude that
\begin{align*}
\lambda_{1,p,q}\left(M\right) \leq  \left(n+1\right)^{-\frac{1}{2}q\left(q+1\right)}\left(\frac{\alpha +1}{p}\int_{M}|d\psi|^{p} dv + \frac{\beta +1}{q}\int_{M} |d\eta|^{q} dv \right).
\end{align*}
In the case that $1< q<2\leq p$, since $x \rightarrow x^{\frac{p}{2}}$ is convex, in the similar way
\begin{align*}
\sum_{i=1}^{n+1}|\check{\psi}_{i}|^{\alpha +1}|\check{\eta}_{i}|^{\beta +1} &\geq \frac{1}{n+1}\sum_{i=1}^{n+1}|\check{\psi}_{i}|^{\alpha +1}|\check{\eta}_{i}|^{\beta +1}\\
&\geq \left(\frac{1}{n+1}\sum_{i=1}^{n+1}|\check{\psi}_{i}|^{\alpha +1}\right) \left(\frac{1}{n+1}\sum_{i=1}^{n+1}|\check{\eta}_{i}|^{\beta +1}\right) \nonumber \\
&\geq \left(\frac{1}{n+1}\sum_{i=1}^{n+1}|\check{\psi}_{i}|^{p\left(\alpha +1\right)}\right) \left(\frac{1}{n+1}\sum_{i=1}^{n+1}|\check{\eta}_{i}|^{p\left(\beta +1\right)}\right) \nonumber \\
&= \left(\sum_{i=1}^{n+1} \frac{1}{n+1}\left(|\check{\psi}_{i}|^{2}\right)^{\frac{p}{2}\left(\alpha +1\right)}\right) \left(\sum_{i=1}^{n+1} \frac{1}{n+1}\left(|\check{\eta}_{i}|^{2}\right)^{\frac{p}{2}\left(\beta +1\right)}\right),
\end{align*}
and again by Jensen's inequality
\begin{align*}
\sum_{i=1}^{n+1}|\check{\psi}_{i}|^{\alpha +1}|\check{\eta}_{i}|^{\beta +1} \geq \left(n+1\right)^{-\frac{p}{2}\left(\alpha + \beta + 2\right)}.
\end{align*}
Furthermore, since $x \rightarrow x^{\frac{q}{2}}$ is convex, we get
\begin{align*}
\sum_{i=1}^{n+1}|d\check{\psi}_{i}|^{p} \leq \sum_{i=1}^{n+1}|d\check{\psi}_{i}|^{q} = \sum_{i=1}^{n+1}\left(|d\check{\psi}_{i}|^{2}\right)^{\frac{q}{2}} \leq \left(\sum_{i=1}^{n+1}|d\check{\psi}_{i}|^{2}\right)^{\frac{q}{2}} =\left(n+1\right)^{1-\frac{q}{2}} |d\psi |^{q},
\end{align*}
and
\begin{align*}
\sum_{i=1}^{n+1}|d\check{\eta}_{i}|^{q} = \sum_{i=1}^{n+1}\left(|d\check{\eta}_{i}|^{2}\right)^{\frac{q}{2}} \leq \left(\sum_{i=1}^{n+1}|d\check{\eta}_{i}|^{2}\right)^{\frac{q}{2}} =\left(n+1\right)^{1-\frac{q}{2}} |d\eta |^{q}.
\end{align*}
These together imply that
\begin{align*}
\lambda_{1,p,q}\left(M\right) \leq \left(n+1\right)^{\frac{1}{2}\left(p^{2}-q\right)}\left( \frac{\alpha +1}{p}\int_{M} |d\psi |^{q} dv + \frac{\beta +1}{q} \int_{M}|d\eta |^{q} dv \right).
\end{align*}
\end{proof}
\begin{proof}[{\bf Proof of Theorem \ref{yet}}]
For $p,q \geq 2$, by Lemma \ref{lemyek} we saw that
\begin{align*}
\lambda_{1,p,q}\left(M\right) \leq \left(n+1\right)^{\frac{1}{2}p^{2}}\left(\frac{\alpha +1}{p}\int_{M}|d\psi|^{p} + \frac{\beta +1}{q}\int_{M}|d\eta|^{q} \right),
\end{align*}
and also
\begin{align*}
\int_{M}|d\psi|^{p} dv &\leq \left(\int_{M}|d\psi|^{m} dv\right)^{\frac{p}{m}}\\
\int_{M}|d\eta|^{q} dv &\leq \left(\int_{M}|d\eta|^{m} dv\right)^{\frac{q}{m}},
\end{align*}
on the other side, $\psi = \gamma \circ \phi : \left(M,g\right) \rightarrow \left(\mathbb{S}^{n}, can\right)$ is a conformal immersion and since
\begin{align*}
\left(\gamma \circ \phi \right)^{*}can &= \frac{|d\left(\gamma \circ \phi \right)|^{2}}{m} = \frac{|d\psi |^{2}}{m},
\end{align*}
from \cite{ma2} we conclude that
\begin{align*}
\int_{M}|d\psi|^{p} dv &= m^{\frac{p}{2}} {\rm vol}\left(M, \left(\gamma\circ \phi \right)^{*}can \right) \\
&\leq m^{\frac{p}{2}} \sup_{\gamma \in G\left(n\right)}{\rm vol}\left(M, \left(\gamma \circ \phi\right)^{*}can \right),
\end{align*}
and in the similar way
\begin{align*}
\int_{M}|d\eta|^{q} dv &= m^{\frac{q}{2}} {\rm vol}\left(M, \left(\delta \circ \phi \right)^{*}can \right) \\
&\leq m^{\frac{q}{2}} \sup_{\delta \in G\left(n\right)}{\rm vol}\left(M, \left(\delta \circ \phi\right)^{*}can \right).
\end{align*}
Now by taking "$\inf $" with respect to  $\phi$ in the above inequality we get
\begin{align*}
\lambda_{1,p,q}\left(M\right) \leq \left(n+1\right)^{\frac{1}{2}p^{2}} \Big[\frac{\alpha +1}{p}m^{\frac{p}{2}}\left(V_{n}^{c}\left(M,[g]\right)\right)^{\frac{p}{m}} + \frac{\beta +1}{q}m^{\frac{q}{2}}\left(V_{n}^{c}\left(M, [g]\right)\right)^{\frac{q}{m}}\Big].
\end{align*}
Since $p \geq q$, we have
\begin{align*}
\lambda_{1,p,q}\left(M\right) \leq \left(n+1\right)^{\frac{1}{2}p^{2}}m^{\frac{p}{2}}\left(V_{n}^{c}\left(M,[g]\right)\right)^{\frac{p}{m}}.
\end{align*}
Also for $1< p,q <2$, in the similar context we get
\begin{align*}
\lambda_{1,p,q}\left(M\right) \leq \left(n+1\right)^{-\frac{1}{2}q\left(q+1\right)}m^{\frac{q}{2}}\left(V_{n}^{c}\left(M,[g]\right)\right)^{\frac{q}{m}},
\end{align*}
and also for $1<q<2\leq p$ we find that
\begin{align*}
\lambda_{1,p,q}\left(M\right) \leq \left(n+1\right)^{\frac{1}{2}\left(p^{2} - q\right)}m^{\frac{q}{2}}\left(V_{n}^{c}\left(M,[g]\right)\right)^{\frac{q}{m}}.
\end{align*}
\end{proof}
The similar problem was studied before in \cite{li} for surfaces and also for upper dimension manifold there are some results in \cite{el} for $p$-Laplacian operator. Li and Yau \cite{li}, proved that the upper bound for $V_{n}^{c}\left(M, [g]\right)$ just depend on the genus of $M$. They actually proved that for orientable surface $M$ (when $m=2$) and for $n\geq 2$ we have
\begin{align*}
V_{n}^{c}\left(M,[g]\right) \leq 4\pi \Big[\frac{\tau \left(M\right) +3}{2}\Big],
\end{align*}
and also for non-orientable surface $M$ and $n\geq 4$ we get
\begin{align*}
V_{n}^{c}\left(M,[g]\right) \leq 12\pi \Big[\frac{\tau \left(M\right) +3}{2}\Big],
\end{align*}
where $\tau\left(M\right)$ is genus of $M$ and $[.]$ denotes the bracket function.
\begin{Rem}
Consider $M$ as a compact manifold and $m\geq p \geq q$. Let $\lambda_{1,p,q}$ denotes the first eigenvalue of the $\left(p,q\right)$-Laplacian (\ref{se}), if $M$ is orientable and $n\geq 2$, then
\begin{itemize}
\item
If $p\geq q\geq 2$, then
\begin{align*}
\lambda_{1,p,q}\left(M\right) \leq \left(n+1\right)^{\frac{1}{2}p^{2}}\left(8\pi \right)^{\frac{p}{2}} \Big[\frac{\tau \left(M\right) +3}{2}\Big]^{\frac{p}{2}}.
\end{align*}
\item
If $1<q\leq p <2$, then
\begin{align*}
\lambda_{1,p,q}\left(M\right) \leq \left(n+1\right)^{-\frac{1}{2}q\left(q+1\right)}\left(8\pi \right)^{\frac{p}{2}} \Big[\frac{\tau \left(M\right) +3}{2}\Big]^{\frac{p}{2}}.
\end{align*}
\item
If $1<q<2\leq p$, then
\begin{align*}
\lambda_{1,p,q}\left(M\right) \leq \left(n+1\right)^{\frac{1}{2}\left(p^{2}-q\right)}\left(8\pi \right)^{\frac{p}{2}} \Big[\frac{\tau \left(M\right) +3}{2}\Big]^{\frac{p}{2}}.
\end{align*}
\end{itemize}
Also if $M$ is non-orientable and for $n\geq 4$,
\begin{itemize}
\item
if $p\geq q\geq 2$, then
\begin{align*}
\lambda_{1,p,q}\left(M\right) \leq \left(n+1\right)^{\frac{1}{2}p^{2}}\left(24\pi \right)^{\frac{p}{2}} \Big[\frac{\tau \left(M\right) +3}{2}\Big]^{\frac{p}{2}}.
\end{align*}
\item
If $1<q\leq p<2$, then
\begin{align*}
\lambda_{1,p,q}\left(M\right) \leq \left(n+1\right)^{-\frac{1}{2}q\left(q+1\right)}\left(24\pi \right)^{\frac{p}{2}} \Big[\frac{\tau \left(M\right) +3}{2}\Big]^{\frac{p}{2}}.
\end{align*}
\item
If $1<q<2\leq p$, then
\begin{align*}
\lambda_{1,p,q}\left(M\right) \leq \left(n+1\right)^{\frac{1}{2}\left(p^{2}-q\right)}\left(24\pi \right)^{\frac{p}{2}} \Big[\frac{\tau \left(M\right) +3}{2}\Big]^{\frac{p}{2}}.
\end{align*}
\end{itemize}
\end{Rem}
\section{The second case, $p,q >m$}
Let $r \in \left[0, \pi\right]$ be a geodesic distance and $\epsilon >0$, then the radial function $f_{\epsilon}: \mathbb{S}^{n} \rightarrow \mathbb{R}$ is defined as
\begin{align*}
f_{\epsilon}\left(r\right) = \epsilon^{\frac{4p}{m\left(p-m\right)}}. \chi_{[0, \frac{\pi}{2}-\epsilon ] \cup [\frac{\pi}{2}+\epsilon, \pi ]}\left(r\right) + \chi_{\left(\frac{\pi}{2}-\epsilon, \frac{\pi}{2}+\epsilon \right)}\left(r\right),
\end{align*}
where $\chi$ is denoted as characteristic function. Now let
\begin{align*}
R_{\epsilon}\left(u,v\right):=   \frac{1}{\int_{\mathbb{S}^{m-1}}f_{\epsilon}^{\frac{m}{2}}|u|^{\alpha +1}|v|^{\beta +1}dv_{can}} &\Big[\frac{\alpha +1}{p}\int_{\mathbb{S}^{m-1}}f_{\epsilon}^{\frac{m-p}{2}}|d u|^{p}dv_{can}\\
&+ \frac{\beta +1}{q}\int_{\mathbb{S}^{m-1}}f_{\epsilon}^{\frac{m-p}{2}}|d v|^{q}dv_{can} \Big].
\end{align*}
Then
\begin{align*}
\lambda_{1,p,q}\left(\epsilon \right) = \inf_{u,v \neq 0} \Big\lbrace R_{\epsilon}\left(u,v\right)| \left(u,v\right) \in W_{0}^{1,p}\times W_{0}^{1,q} \setminus \lbrace 0 \rbrace \Big\rbrace.
\end{align*}
It seems clear that $\lambda_{1,p,q}\left(\epsilon\right)$ is a parametrization for the first eigenvalue of the $\left(p,q\right)$-Laplacian system (\ref{se}).
\begin{Th}\label{dot}
Consider $M$ as an $m$-dimensional compact manifold. If $\lambda_{1,p,q}$ denotes the first eigenvalue of the $\left(p,q\right)$-Laplacian system (\ref{se}) and $p\geq q >m$ then
\begin{align*}
\limsup_{\epsilon \rightarrow 0}\lambda_{1,p,q}\left(\epsilon \right). \epsilon^{\frac{p}{m}} = \infty.
\end{align*}
\end{Th}
This theorem actually gives us the comparison between $\lambda_{1,p,q}\left(\epsilon\right)$ and $\lambda_{1,p,q}\left(\mathbb{S}^{m}, can\right)$.
\begin{proof}[{\bf Proof of Theorem \ref{dot}}]
Consider the radial functions $\bar{u}_{\epsilon}, \bar{v}_{\epsilon}: \mathbb{S}^{m} \rightarrow \mathbb{R}$ as
\begin{align*}
\bar{u}_{\epsilon}^{p}\left(r\right) &= \frac{1}{V}\int_{\mathbb{S}^{m-1}}|u_{\epsilon}\left(r,.\right)|^{p} dv_{can},\\
\bar{v}_{\epsilon}^{q}\left(r\right) &= \frac{1}{V}\int_{\mathbb{S}^{m-1}}|v_{\epsilon}\left(r,.\right)|^{q} dv_{can},
\end{align*}
where $V$ stands for ${\rm vol}\left(\mathbb{S}^{m-1}, can\right)$. By taking derivative with respect to $r$ we get
\begin{align*}
p\bar{u}_{\epsilon}^{p-1}\bar{u}^{\prime}_{\epsilon} = \frac{p}{V}\int_{\mathbb{S}^{m-1}}|u_{\epsilon}|^{p-2}u_{\epsilon} \frac{\partial u_{\epsilon}}{\partial r} dv_{can},
\end{align*}
and the similar context holds for $v$ and $q$. By H\"older's inequality we have
\begin{align*}
\bar{u}_{\epsilon}^{p-1}|\bar{u}_{\epsilon}^{\prime}| &\leq \frac{1}{V}\int_{\mathbb{S}^{m-1}}|u_{\epsilon}|^{p-1}|\frac{\partial u_{\epsilon}}{\partial r}| dv_{can}\\
&\leq \frac{1}{V}\left(\int_{\mathbb{S}^{m-1}}|u_{\epsilon}|^{p} dv_{can}\right)^{\frac{p-1}{p}} \left(\int_{\mathbb{S}^{m-1}}|\frac{\partial u_{\epsilon}}{\partial r}|^{p} dv_{can} \right)^{\frac{1}{p}},
\end{align*}
which concludes that
\begin{align}\label{pang}
|\bar{u}^{\prime}_{\epsilon}|^{p} \leq \frac{1}{V} \int_{\mathbb{S}^{m-1}}|\frac{\partial u_{\epsilon}}{\partial r}|^{p} dv_{can} \leq \frac{1}{V}\int_{\mathbb{S}^{m-1}}|du_{\epsilon}|^{p} dv_{can}.
\end{align}
Since $\frac{\alpha +1}{2} + \frac{\beta +1}{2} =1$, H\"older's inequality implies that
\begin{align*}
\int_{\mathbb{S}^{m}} |u_{\epsilon}|^{\alpha +1}|v_{\epsilon}|^{\beta +1}dv_{can} \leq \left(\int_{\mathbb{S}^{m}}|u_{\epsilon}|^{p} dv_{can}\right)^{\frac{\alpha +1}{p}}\left(\int_{\mathbb{S}^{m}}|v_{\epsilon}|^{q} dv_{can}\right)^{\frac{\beta +1}{q}},
\end{align*}
and again by H\"older's inequality we see
\begin{center}
\begin{align*}
\int_{\mathbb{S}^{m}}f_{\epsilon}^{\frac{m}{2}}|\bar{u}_{\epsilon}|^{\alpha +1}|\bar{v}_{\epsilon}|^{\beta +1} dv_{can} &= V.\int_{0}^{\pi}f_{\epsilon}^{\frac{m}{2}}|\bar{u}_{\epsilon}|^{\alpha +1}|\bar{v}_{\epsilon}|^{\beta +1} \sin r^{m-1} dr \\ \nonumber
&= V.\int_{0}^{\pi}\Big[ f_{\epsilon}^{\frac{m}{2}}\left(\frac{1}{V}\int_{\mathbb{S}^{m-1}}|u_{\epsilon}|^{p} dv_{can}\right)^{\frac{\alpha +1}{p}}\\
&\left(\frac{1}{V}\int_{\mathbb{S}^{m-1}}|v_{\epsilon}|^{q} dv_{can}\right)^{\frac{\beta +1}{q}} \sin r^{m-1} \Big]dr \\ \nonumber
&=\int_{0}^{\pi}\Big[ f_{\epsilon}^{\frac{m}{2}}\left(\int_{\mathbb{S}^{m-1}}|u_{\epsilon}|^{p} dv_{can}\right)^{\frac{\alpha +1}{p}}\\
&\left(\int_{\mathbb{S}^{m-1}}|v_{\epsilon}|^{q} dv_{can}\right)^{\frac{\beta +1}{q}} \sin r^{m-1} \Big]dr \\ \nonumber
&\geq \int_{0}^{\pi} f_{\epsilon}^{\frac{m}{2}}\left(\int_{\mathbb{S}^{m-1}}|u_{\epsilon}|^{\alpha +1}|v_{\epsilon}|^{\beta +1} dv_{can}\right) \sin r^{m-1} dr \\ \nonumber
&\geq \int_{\mathbb{S}^{m}}f_{\epsilon}^{\frac{m}{2}}|u_{\epsilon}|^{\alpha +1}|v_{\epsilon}|^{\beta +1} dv_{can}.
\end{align*}
\end{center}
From (\ref{pang}), we get
\begin{align*}
\int_{\mathbb{S}^{m}}f_{\epsilon}^{\frac{m-p}{2}}|\bar{u}_{\epsilon}^{\prime}|^{p} dv_{can} &= V.\int_{0}^{\pi}f_{\epsilon}^{\frac{m-p}{2}}|\bar{u}_{\epsilon}^{\prime}|^{p} \sin r^{m-1} dr \\
&\leq \int_{0}^{\pi} \Big[\int_{\mathbb{S}^{m-1}}|du_{\epsilon}|^{p} dv_{can} \Big]f_{\epsilon}^{\frac{m-p}{2}} \sin r^{m-1} dr \\
&= \int_{\mathbb{S}^{m}}f_{\epsilon}^{\frac{m-p}{2}}|du_{\epsilon}|^{p} dv_{can},
\end{align*}
and in the similar way
\begin{align*}
\int_{\mathbb{S}^{m}}f_{\epsilon}^{\frac{m-p}{2}}|\bar{v}_{\epsilon}^{\prime}|^{q} dv_{can} \leq \int_{S^{m}}f_{\epsilon}^{\frac{m-p}{2}}|dv_{\epsilon}|^{q} dv_{can}.
\end{align*}
If $\mathbb{S}^{m}_{+}$ and $\mathbb{S}^{m}_{-}$ denote the upper and lower hemispheres centered at $x_{0}$ and $-x_{0}$ respectively, then
\begin{align*}
\lambda_{1,p,q}\left(\epsilon \right) &\geq \frac{1}{\int_{\mathbb{S}^{m}}f_{\epsilon}^{\frac{m}{2}}|\bar{u}_{\epsilon}|^{\alpha +1}|\bar{v}_{\epsilon}|^{\beta +1}dv_{can}} \Big[ \frac{\alpha +1}{p}\int_{\mathbb{S}^{m}}f_{\epsilon}^{\frac{m-p}{2}}|\bar{u}_{\epsilon}^{\prime}|^{p} dv_{can}
+ \frac{\beta +1}{q} \int_{\mathbb{S}^{m}}f_{\epsilon}^{\frac{m-p}{2}}|\bar{v}_{\epsilon}^{\prime}|^{q} dv_{can} \Big] \\
&\geq \min \lbrace \lambda_{1,p,q}^{+}, \lambda_{1,p,q}^{-} \rbrace,
\end{align*}
where $\lambda_{1,p,q}^{+}$ and $\lambda_{1,p,q}^{-}$  mean that taking above integrals on upper and lower hemispheres respectively. Without loss of generality, let
\begin{align*}
\lambda_{1,p,q}\left(\epsilon \right) \geq \lambda_{1,p,q}^{+}\left(\epsilon \right),
\end{align*}
which means
\begin{align*}
\lambda_{1,p,q}\left(\epsilon \right) &\geq \frac{1}{\int_{\mathbb{S}^{m}_{+}}f_{\epsilon}^{\frac{m}{2}}|\bar{u}_{\epsilon}|^{\alpha +1}|\bar{v}_{\epsilon}|^{\beta +1}dv_{can}} \Big[ \frac{\alpha +1}{p}\int_{\mathbb{S}^{m}_{+}}f_{\epsilon}^{\frac{m-p}{2}}|\bar{u}_{\epsilon}^{\prime}|^{p} dv_{can}
+ \frac{\beta +1}{q} \int_{\mathbb{S}^{m}_{+}}f_{\epsilon}^{\frac{m-p}{2}}|\bar{v}_{\epsilon}^{\prime}|^{q} dv_{can} \Big].
\end{align*}
Now consider two functions $a_{\epsilon} \in W^{1,p}\left(M\right)$ and $c_{\epsilon} \in W^{1,q}\left(M\right)$ as
\begin{equation*}
a_{\epsilon} = \left\{
\begin{array}{lr}
\bar{u}_{\epsilon}   &  [0, \frac{\pi}{2}-\epsilon ] , \\
\bar{u}_{\epsilon}\left(\frac{\pi}{2}-\epsilon \right)  &  \left(\frac{\pi}{2}-\epsilon, \frac{\pi}{2}\right] ,
\end{array} \right.
\end{equation*}
and
\begin{equation*}
c_{\epsilon} = \left\{
\begin{array}{lr}
\bar{v}_{\epsilon}   &  [0, \frac{\pi}{2}-\epsilon ] , \\
\bar{v}_{\epsilon}\left(\frac{\pi}{2}-\epsilon \right)  &  \left(\frac{\pi}{2}-\epsilon, \frac{\pi}{2}\right] ,
\end{array} \right.
\end{equation*}
also let $b_{\epsilon} = \bar{u}_{\epsilon} - a_{\epsilon}$ and $d_{\epsilon} = \bar{v}_{\epsilon} - c_{\epsilon}$. Obviously, on $\left[0, \frac{pi}{2} - \epsilon\right]$ and $\left(\frac{\pi}{2} - \epsilon, \frac{\pi}{2}\right]$ we have $b_{\epsilon} = d_{\epsilon} =0$ and $a_{\epsilon} = c_{\epsilon} =0$ respectively. From above definitions we see
\begin{align*}
|\bar{u}_{\epsilon}^{\prime}|^{p} &= |a_{\epsilon}^{\prime}|^{p} + |b_{\epsilon}^{\prime}|^{p},\\
|\bar{v}_{\epsilon}^{\prime}|^{q} &= |c_{\epsilon}^{\prime}|^{q} + |d_{\epsilon}^{\prime}|^{q},\\
|\bar{u}_{\epsilon}|^{\alpha +1} &\leq 2^{\alpha} \left(|a_{\epsilon}|^{\alpha +1} + |b_{\epsilon}|^{\alpha +1}\right),\\
|\bar{v}_{\epsilon}|^{\beta +1} &\leq 2^{\beta}\left(|c_{\epsilon}|^{\beta +1} + |d_{\epsilon}|^{\beta +1}\right).
\end{align*}
And also from definition of $f_{\epsilon}\left(r\right)$ and substituting in formulae of $\lambda_{1,p,q}$ we get
\begin{align*}
\lambda_{1,p,q}\left(\epsilon \right) \geq \frac{2^{-\left(\alpha + \beta \right)}}{\mathcal{A}} &\Big[\epsilon^{-\frac{2p}{m}}\left(\frac{\alpha +1}{p}\int_{\mathbb{S}_{+}^{m}}|a_{\epsilon}^{\prime}|^{p} dv_{can} + \frac{\beta +1}{q}\int_{\mathbb{S}_{+}^{m}}|c_{\epsilon}^{\prime}|^{q} dv_{can} \right) \\
&+ \frac{\alpha +1}{p}\int_{\mathbb{S}_{+}^{m}}|b_{\epsilon}^{\prime}|^{p} dv_{can} + \frac{\beta +1}{q}\int_{\mathbb{S}_{+}^{m}}|d_{\epsilon}^{\prime}|^{q} dv_{can} \Big],
\end{align*}
where
\begin{align*}
\mathcal{A} &= \int_{\mathbb{S}_{+}^{m}}f_{\epsilon}^{\frac{m}{2}}|a_{\epsilon}|^{\alpha +1}|c_{\epsilon}|^{\beta +1} dv_{can} + \int_{\mathbb{S}_{+}^{m}}|a_{\epsilon}|^{\alpha +1}|d_{\epsilon}|^{\beta +1} dv_{can} + \int_{\mathbb{S}_{+}^{m}}|b_{\epsilon}|^{\alpha +1}|c_{\epsilon}|^{\beta +1} dv_{can} \\
&+ \int_{\mathbb{S}_{+}^{m}}|b_{\epsilon}|^{\alpha +1}|d_{\epsilon}|^{beta +1} dv_{can}.
\end{align*}
Now let
\begin{align*}
\mathcal{A}=1,
\end{align*}
so obviously,
\begin{align}\label{shesh}
\lambda_{1,p,q}\left(\epsilon \right) \geq 2^{-\left(\alpha +\beta \right)} &\Big[\epsilon^{-\frac{2p}{m}}\left(\frac{\alpha +1}{p}\int_{\mathbb{S}^{m}_{+}}|a_{\epsilon}^{\prime}|^{p} dv_{can} + \frac{\beta +1}{q}\int_{\mathbb{S}^{m}_{+}}|c_{\epsilon}^{\prime}|^{q} dv_{can} \right) \\
&+ \left(\frac{\alpha +1}{p}\int_{\mathbb{S}_{m}^{+}}|b_{\epsilon}^{\prime}|^{p} dv_{can} + \frac{\beta +1}{q}\int_{\mathbb{S}_{m}^{+}}|d_{\epsilon}^{\prime}|^{q} dv_{can}\right)\Big]. \nonumber
\end{align}
We consider two different cases, on the one hand,
\begin{align*}
\limsup_{\epsilon \rightarrow 0} \Big[\frac{\alpha +1}{p}\int_{\mathbb{S}^{m}_{+}}|a_{\epsilon}^{\prime}|^{p} dv_{can} + \frac{\beta +1}{q}\int_{\mathbb{S}^{m}_{+}}|c_{\epsilon}^{\prime}|^{q} dv_{can} \Big] >0,
\end{align*}
then
\begin{align*}
\lambda_{1,p,q}\left(\epsilon \right). \epsilon^{\frac{p}{m}} \geq 2^{-\left(\alpha + \beta \right)}\epsilon^{-\frac{p}{m}}\left(\frac{\alpha +1}{p}\int_{\mathbb{S}^{m}_{+}}|a_{\epsilon}^{\prime}|^{p} dv_{can} + \frac{\beta +1}{q}\int_{\mathbb{S}^{m}_{+}}|c_{\epsilon}^{\prime}|^{q} dv_{can}\right),
\end{align*}
which concludes that
\begin{align*}
\limsup_{\epsilon \rightarrow 0}\lambda_{1,p,q}\left(\epsilon \right).\epsilon^{\frac{p}{m}} = \infty.
\end{align*}
On the other hand,
\begin{align*}
\lim_{\epsilon \rightarrow 0} \Big[\frac{\alpha +1}{p}\int_{\mathbb{S}^{m}_{+}}|a_{\epsilon}^{\prime}|^{p} dv_{can} + \frac{\beta +1}{q}\int_{\mathbb{S}^{m}_{+}}|c_{\epsilon}^{\prime}|^{q} dv_{can}\Big] =0,
\end{align*}
then we choose the sequence $\epsilon_{n} \rightarrow 0$ in the case that $a_{\epsilon_{n}} + c_{\epsilon_{n}} \rightarrow a+c$ where $a$ and $c$ are real constants. Now since
\begin{align*}
\lim_{n \rightarrow \infty}\int_{\mathbb{S}^{m}_{+}}f_{\epsilon_{n}}^{\frac{m}{2}}|a_{\epsilon_{n}}|^{\alpha +1}|c_{\epsilon_{n}}|^{\beta +1} dv_{can} &= \lim_{n \rightarrow \infty}\int_{\mathbb{S}^{m}_{+}}f_{\epsilon_{n}}^{\frac{m}{2}}\left(|a_{\epsilon_{n}}|^{\alpha +1}|c_{\epsilon_{n}}|^{\beta +1} - |a|^{\alpha +1}|c|^{\beta +1}\right) dv_{can} \\
&+ \left(|a|^{\alpha +1}|c|^{\beta +1}\right)\lim_{n \rightarrow \infty}\int_{\mathbb{S}_{+}^{m}}f_{\epsilon_{n}}^{\frac{m}{2}}dv_{can} =0,
\end{align*}
and for $p,q >m$, $\lbrace f_{\epsilon_{n}} \rbrace$ is uniformly bounded, thus
\begin{align*}
\lim_{n \rightarrow \infty} \int_{S^{m}_{+}}f_{\epsilon_{n}}^{\frac{m}{2}} dv_{can} = 0.
\end{align*}
By substituting above formulaes in (\ref{shesh}) we get
\begin{align*}
\lambda_{1,p,q}\left(\epsilon \right) &\geq \frac{2^{-\left(\alpha + \beta \right)}}{\mathcal{B}}\Big[\frac{\alpha +1}{p}\int_{\mathbb{S}^{m}_{+}}|b_{\epsilon}^{\prime}|^{p} dv_{can} + \frac{\beta +1}{q}\int_{\mathbb{S}^{m}_{+}}|d_{\epsilon}^{\prime}|^{q} dv_{can} \Big] \\
&= \frac{2^{-\left(\alpha + \beta \right)}}{\int_{\frac{\pi}{2}- \epsilon_{n}}^{\frac{\pi}{2}}\mathcal{D} \sin r^{m-1} dr}\int_{\frac{\pi}{2}-\epsilon_{n}}^{\frac{\pi}{2}}\left(\frac{\alpha +1}{p}|b_{\epsilon}^{\prime}|^{p} + \frac{\beta +1}{q}|d_{\epsilon}^{\prime}|^{q}\right) \sin r^{m-1} dr \\
&\geq 2^{-\left(\alpha + \beta \right)}\left(\sin \left(\frac{\pi}{2} - \epsilon_{n}\right)\right)^{m-1}\frac{\frac{\alpha +1}{p}\int_{\frac{\pi}{2}-\epsilon_{n}}^{\frac{\pi}{2}}|b_{\epsilon}^{\prime}|^{p} dr + \frac{\beta +1}{q}\int_{\frac{\pi}{2}-\epsilon_{n}}^{\frac{\pi}{2}}|d_{\epsilon}^{\prime}|^{q} dr}{\int_{\frac{\pi}{2}-\epsilon_{n}}^{\frac{\pi}{2}} \mathcal{D} dr},
\end{align*}
where
\begin{align*}
\mathcal{B} &= \int_{\mathbb{S}^{m}_{+}}|a_{\epsilon}|^{\alpha +1}|d_{\epsilon}|^{\beta +1} dv_{can} + \int_{\mathbb{S}^{m}_{+}}|b_{\epsilon}|^{\alpha +1}|c_{\epsilon}|^{\beta +1} dv_{can} \\
&+ \int_{\mathbb{S}^{m}_{+}}|b_{\epsilon}|^{\alpha +1}|d_{\epsilon}|^{\beta +1} dv_{can},
\end{align*}
and
\begin{align*}
\mathcal{D} = |a_{\epsilon}|^{\alpha +1}|d_{\epsilon}|^{\beta +1} + |b_{\epsilon}|^{\alpha +1}|c_{\epsilon}|^{\beta +1} + |b_{\epsilon}|^{\alpha +1}|d_{\epsilon}|^{\beta +1}.
\end{align*}
Consider $\bar{a}_{\epsilon_{n}} \in W_{0}^{1,p}\left(-\epsilon, \epsilon\right)$ as
\begin{align*}
\bar{a}_{\epsilon_{n}}\left(x\right) = a_{\epsilon_{n}}\left(x + \frac{\pi}{2} - \epsilon_{n}\right).
\end{align*}
The similar way holds for $\bar{b}_{\epsilon_{n}} \in W_{0}^{1,p}\left(-\epsilon_{n}, \epsilon_{n}\right)$ and $\bar{c}_{\epsilon_{n}}, \bar{d}_{\epsilon_{n}} \in W_{0}^{1,q}\left(-\epsilon, \epsilon\right)$, and also these functions are even, so
\begin{align*}
\frac{\frac{\alpha +1}{p}\int_{\frac{\pi}{2}-\epsilon_{n}}^{\frac{\pi}{2}}|b_{\epsilon}^{\prime}|^{p} dr + \frac{\beta +1}{q}\int_{\frac{\pi}{2}-\epsilon_{n}}^{\frac{\pi}{2}}|d_{\epsilon}^{\prime}|^{q} dr}{\int_{\frac{\pi}{2}-\epsilon_{n}}^{\frac{\pi}{2}}\mathcal{D} dr} &= \frac{\frac{\alpha +1}{p}\int_{0}^{\epsilon_{n}}|b_{\epsilon}^{\prime}|^{p} dr + \frac{\beta +1}{q}\int_{0}^{\epsilon_{n}}|d_{\epsilon}^{\prime}|^{q} dr}{\int_{0}^{\epsilon_{n}} \mathcal{D} dr}\\
&= \frac{\frac{\alpha +1}{p}\int_{-\epsilon_{n}}^{\epsilon_{n}}|b_{\epsilon}^{\prime}|^{p} dr + \frac{\beta +1}{q}\int_{-\epsilon_{n}}^{\epsilon_{n}}|d_{\epsilon}^{\prime}|^{q} dr}{\int_{-\epsilon_{n}}^{\epsilon_{n}}\mathcal{D} dr}\\
&\geq \lambda_{1,p,q}\left(-\epsilon_{n}, \epsilon_{n}\right) \\
&= \epsilon_{n}^{-p}\lambda_{1,p,q}\left(-1, 1\right),
\end{align*}
and
\begin{align*}
\lambda_{1,p,q}\left(\epsilon \right) \geq 2^{-\left(\alpha + \beta \right)} . \epsilon_{n}^{-p}\left(\sin \left(\frac{\pi}{2} - \epsilon_{n}\right)\right)^{m-1} \lambda_{1,p,q}\left(-1, 1\right),
\end{align*}
which concludes finally
\begin{align*}
\limsup_{\epsilon \rightarrow 0} \lambda_{1,p,q}\left(\epsilon \right) . \epsilon^{\frac{p}{m}} = \infty.
\end{align*}
\end{proof}
For $\epsilon >0$, let $\tilde{f}_{\epsilon} \in C^{\infty}\left(\mathbb{S}^{m}\right)$ be a radial function suth that $\tilde{f}_{\epsilon} \leq f_{\epsilon}$. Also on $\left[\frac{\pi}{2} - \frac{\epsilon}{2}, \frac{\pi}{2} + \frac{\epsilon}{2}\right]$ we get
\begin{align*}
\tilde{f}_{\epsilon}\left(r\right) = f_{\epsilon}\left(r\right) =1,
\end{align*}
and
\begin{align*}
\tilde{f}_{\epsilon}\left(\pi - r\right) = \tilde{f}\left(r\right).
\end{align*}
Furthermore
\begin{align*}
{\rm vol}\left(\mathbb{S}^{m}, \tilde{f}_{\epsilon}can\right) = \int_{\mathbb{S}^{m}}\tilde{f}_{\epsilon}^{\frac{m}{2}} dv_{can} &= \int_{\mathbb{S}^{m-1}}\int_{-\frac{\pi}{2}}^{\frac{\pi}{2}}\tilde{f}_{\epsilon}^{\frac{m}{2}} \sin r^{m-1} dr dv_{can} \\
&> V. \int_{\frac{\pi}{2}-\frac{\epsilon}{2}}^{\frac{\pi}{2} + \frac{\epsilon}{2}} \sin r^{m-1} dr \\
&> \epsilon V\Big[\sin \left(\frac{\pi}{2} - \epsilon \right)\Big]^{m-1},
\end{align*}
where $V = {\rm vol}\left(\mathbb{S}^{m}, can\right)$. If $\tilde{u}_{\epsilon}$ and $\tilde{v}_{\epsilon}$ denote the eigenfunctions for $\lambda_{1,p,q}\left(\mathbb{S}^{n}, \tilde{f}_{\epsilon}can\right)$, and $\tilde{u}_{\epsilon}^{+}$, $\tilde{u}_{\epsilon}^{-}$, $\tilde{v}_{\epsilon}^{+}$, $\tilde{v}_{\epsilon}^{-}$ denote the positive and the negative parts of $\tilde{u}_{\epsilon}$ and $\tilde{v}_{\epsilon}$ respectively. For the $p$-Laplacian (\ref{yek}) it was proved before in \cite{ma2} that
\begin{align}\label{haft}
\lambda_{1,p}\left(\mathbb{S}^{m}, \tilde{f}_{\epsilon} can\right) = \frac{\int_{\mathbb{S}^{m}}|d\tilde{u}_{\epsilon}^{+}|^{p} \tilde{f}_{\epsilon}^{\frac{m-p}{2}}dv_{can}}{\int_{\mathbb{S}^{m}}|\tilde{u}_{\epsilon}^{+}|^{p}\tilde{f}_{\epsilon}^{\frac{m}{2}}dv_{can}} =  \frac{\int_{\mathbb{S}^{m}}|d \tilde{u}_{\epsilon}^{-}|^{p} \tilde{f}_{\epsilon}^{\frac{m-p}{2}}dv_{can}}{\int_{\mathbb{S}^{m}}|\tilde{u}_{\epsilon}^{-}|^{p}\tilde{f}_{\epsilon}^{\frac{m}{2}}dv_{can}}.
\end{align}
\begin{Cor}\label{coryek}
Let $p \geq q >m$ and $\lambda_{1,p,q}$ denotes the first eigenvalue for the $\left(p,q\right)$-Laplacian system (\ref{se}) then for $\lambda_{1,p,q}$ arbitrary large, there exists the Riemannian metric with volume one on $\mathbb{S}^{m}$ conformal to the standard metric $can$.
\end{Cor}
\begin{proof}
By expanding (\ref{haft}) on the $\left(p,q\right)$-Laplacian system (\ref{se}) we have
\begin{align*}
\lambda_{1,p,q}\left(\mathbb{S}^{m}, \tilde{f}_{\epsilon} can\right) &= \frac{1}{\int_{\mathbb{S}^{m}}|\tilde{u}_{\epsilon}^{+}|^{\alpha +1}|\tilde{v}_{\epsilon}^{+}
|^{\beta +1} \tilde{f}_{\epsilon}^{\frac{m}{2}} dv_{can}} \Big[\frac{\alpha +1}{p}\int_{\mathbb{S}^{m}}|d\tilde{u}_{\epsilon}^{+}|^{p} \tilde{f}_{\epsilon}^{\frac{m-p}{2}} dv_{can} \\
&+ \frac{\beta +1}{q}\int_{\mathbb{S}^{m}}|d\tilde{v}_{\epsilon}^{+}|^{q} \tilde{f}_{\epsilon}^{\frac{m-p}{2}} dv_{can} \Big] \\
&= \frac{1}{\int_{\mathbb{S}^{m}}|\tilde{u}_{\epsilon}^{-}|^{\alpha +1}|\tilde{v}_{\epsilon}^{-}
|^{\beta +1} \tilde{f}_{\epsilon}^{\frac{m}{2}} dv_{can}} \Big[\frac{\alpha +1}{p}\int_{\mathbb{S}^{m}}|d \tilde{u}_{\epsilon}^{-}|^{p} \tilde{f}_{\epsilon}^{\frac{m-p}{2}} dv_{can} \\
&+ \frac{\beta +1}{q}\int_{\mathbb{S}^{m}}|d\tilde{v}_{\epsilon}^{-}|^{q} \tilde{f}_{\epsilon}^{\frac{m-p}{2}} dv_{can} \Big].
\end{align*}
Let $t \in \mathbb{R}$ such that
\begin{align*}
\tilde{u}_{\epsilon, t} = t\tilde{u}_{\epsilon}^{+} + \tilde{u}_{\epsilon}^{-},
\end{align*}
then
\begin{align*}
\lambda_{1,p,q}\left(\mathbb{S}^{m}, \tilde{f}_{\epsilon} can\right) &= \frac{1}{\int_{\mathbb{S}^{m}}|\tilde{u}_{\epsilon}|^{\alpha +1}|\tilde{v}_{\epsilon}
|^{\beta +1} \tilde{f}_{\epsilon}^{\frac{m}{2}} dv_{can}} \Big[\frac{\alpha +1}{p}\int_{\mathbb{S}^{m}}|d\tilde{u}_{\epsilon}|^{p} \tilde{f}_{\epsilon}^{\frac{m-p}{2}} dv_{can} \\
&+ \frac{\beta +1}{q}\int_{\mathbb{S}^{m}}|d\tilde{v}_{\epsilon}|^{q} \tilde{f}_{\epsilon}^{\frac{m-p}{2}} dv_{can} \Big] \\
&\geq \frac{1}{\int_{\mathbb{S}^{m}}|\tilde{u}_{\epsilon}|^{\alpha +1}|\tilde{v}_{\epsilon}
|^{\beta +1} f_{\epsilon}^{\frac{m}{2}} dv_{can}} \Big[\frac{\alpha +1}{p}\int_{\mathbb{S}^{m}}|d\tilde{u}_{\epsilon}|^{p} f_{\epsilon}^{\frac{m-p}{2}} dv_{can} \\
&+ \frac{\beta +1}{q}\int_{\mathbb{S}^{m}}|d\tilde{v}_{\epsilon}|^{q} f_{\epsilon}^{\frac{m-p}{2}} dv_{can} \Big] \\
&\geq \lambda_{1,p,q}\left(\epsilon \right).
\end{align*}
Above inequalities with the Theorem \ref{dot} and $p \geq q$ together give us
\begin{align*}
\limsup_{\epsilon \rightarrow 0} \lambda_{1,p,q}\left(\mathbb{S}^{m}, \tilde{f}_{\epsilon} can\right) {\rm vol}\left(\mathbb{S}^{m}, \tilde{f}_{\epsilon}can\right)^{\frac{p}{m}} \geq V^{\frac{p}{m}}.\limsup_{\epsilon \rightarrow 0}\lambda_{1,p,q}\left(\epsilon \right). \epsilon^{\frac{p}{m}} = \infty.
\end{align*}
Now set
\begin{align*}
h_{\epsilon} = {\rm vol}\left(\mathbb{S}^{m}, \tilde{f}_{\epsilon}can \right)^{-\frac{2}{m}}\tilde{f}_{\epsilon},
\end{align*}
then we get
\begin{align*}
{\rm vol}\left(S^{m}, h_{\epsilon}can\right) = 1,
\end{align*}
and
\begin{align*}
\limsup_{\epsilon \rightarrow 0}\lambda_{1,p,q}\left(\mathbb{S}^{m}, h_{\epsilon}can\right) = \infty.
\end{align*}
\end{proof}
\begin{Rem}
Someone may consider the situation $q <m<p$, in this case we just take the radial function $f_{\epsilon}: \mathbb{S}^{m} \rightarrow \mathbb{R}$ as
\begin{align*}
f_{\epsilon}\left(r\right) = \epsilon^{\frac{4q}{m\left(m-q\right)}}. \chi_{[0, \frac{\pi}{2}-\epsilon ] \cup [\frac{\pi}{2}+\epsilon, \pi ]}\left(r\right) + \chi_{\left(\frac{\pi}{2}-\epsilon, \frac{\pi}{2}+\epsilon \right)}\left(r\right),
\end{align*}
and then
\begin{align*}
R_{\epsilon}\left(u,v\right):=   \frac{1}{\int_{\mathbb{S}^{m-1}}f_{\epsilon}^{\frac{m}{2}}|u|^{\alpha +1}|v|^{\beta +1}dv_{can}} &\Big[\frac{\alpha +1}{p}\int_{\mathbb{S}^{m-1}}f_{\epsilon}^{\frac{q-m}{2}}|d u|^{p}dv_{can}\\
&+ \frac{\beta +1}{q}\int_{\mathbb{S}^{m-1}}f_{\epsilon}^{\frac{q-m}{2}}|d v|^{q}dv_{can} \Big],
\end{align*}
where
\begin{align*}
\lambda_{1,p,q}\left(\epsilon \right) = \inf_{u,v \neq 0} \Big\lbrace R_{\epsilon}\left(u,v\right)| \left(u,v\right) \in W_{0}^{1,p}\times W_{0}^{1,q} \setminus \lbrace 0 \rbrace \Big\rbrace.
\end{align*}
Now by the definition of $\bar{u}_{\epsilon}$ and $\bar{v}_{\epsilon}$ in the Theorem \ref{dot} and by H\"older's inequality we see
\begin{align*}
\int_{\mathbb{S}^{m}}f_{\epsilon}^{\frac{m}{2}}|\bar{u}_{\epsilon}|^{\alpha +1} dv_{can} = \int_{\mathbb{S}^{m}}f_{\epsilon}^{\frac{m}{2}} |u_{\epsilon}|^{\alpha +1} dv_{can},
\end{align*}
and
\begin{align*}
\int_{\mathbb{S}^{m}}f_{\epsilon}^{\frac{q-m}{2}}|\bar{u}^{\prime}|^{p} dv_{can} \leq \int_{\mathbb{S}^{m}}f_{\epsilon}^{\frac{q-m}{2}}|du|^{p} dv_{can},
\end{align*}
also the same way holds for $v$ and $q$. These together give us
\begin{align*}
\lambda_{1,p,q}\left(\epsilon \right) &\geq \frac{1}{\int_{\mathbb{S}^{m}}f_{\epsilon}^{\frac{m}{2}}|\bar{u}_{\epsilon}|^{\alpha +1}|\bar{v}_{\epsilon}|^{\beta +1}dv_{can}} \Big[ \frac{\alpha +1}{p}\int_{\mathbb{S}^{m}}f_{\epsilon}^{\frac{q-m}{2}}|\bar{u}_{\epsilon}^{\prime}|^{p} dv_{can}\\
&+ \frac{\beta +1}{q} \int_{\mathbb{S}^{m}}f_{\epsilon}^{\frac{q-m}{2}}|\bar{v}_{\epsilon}^{\prime}|^{q} dv_{can} \Big] \\
&\geq \min \lbrace \lambda_{1,p,q}^{+}, \lambda_{1,p,q}^{-} \rbrace.
\end{align*}
So by the same way as Theorem \ref{dot} we get
\begin{align*}
\limsup_{\epsilon \rightarrow 0} \lambda_{1,p,q}\left(\epsilon \right). \epsilon^{\frac{q}{m}} = \infty.
\end{align*}
It means that the same context as Corollary \ref{coryek} holds in the case $q<m<p$.
\end{Rem}
\section{The $\left(p,q\right)$-Laplacian equation, a new characterization}
One may consider the $\left(p,q\right)$-Laplacian equation as
\begin{align*}
\Delta_{p}u + \Delta_{q}u = div\left(\left(|\nabla u|^{p-2} + |\nabla u|^{q-2}\right)\nabla u \right),
\end{align*}
for $1<q<p<\infty$ and also
\begin{align*}
u \in W_{0}^{1,p}\left(M\right) \cap W_{0}^{1,q}\left(M\right).
\end{align*}
Also it can be written as
\begin{align}\label{hasht}
-\Delta_{p}u - \Delta_{q}u = \lambda |u|^{p-2}u,
\end{align}
where for arbitrary $v \in W_{0}^{1,p} \cap W_{0}^{1,q}$, it is equivalent to
\begin{align*}
&\int_{M} |\nabla u|^{p-2} \nabla u \nabla v d\mu + \int_{M} |\nabla u|^{q-2} \nabla u \nabla v dv \\
&=\lambda \int_{M} |u|^{p-2} uv dv,
\end{align*}
and $\lambda $ is called its eigenvalue associated to the eigenvector $u$. Similar to the previous one, in this case the first Dirichlet eigenvalue of the $\left(p,q\right)$-Laplacian equation (\ref{hasht}) is defined as
\begin{align*}
\lambda_{1,p,q}^{D}\left(M\right) = \inf_{u \neq 0} \Big\lbrace \frac{1}{\int_{M}|u|^{p}dv}\left(\int_{M}|\nabla u|^{p}dv + \int_{M}|\nabla u|^{q}dv\right)| u \in W_{0}^{1,p}\left(M\right) \cap W_{0}^{1,q}\left(M\right) \setminus \lbrace 0 \rbrace \Big\rbrace.
\end{align*}
\begin{Th}\label{set}
Consider $M$ as an $m$-dimensional compact manifold and $1<q<p\leq m$. If $\lambda_{1,p,q}^{D}$ denotes the first eigenvalue of the $\left(p,q\right)$-Laplacian equation (\ref{hasht}), then
\begin{align*}
\lambda_{1,p,q}^{D}\left(M\right) \leq \left(n+1\right)^{|\frac{p}{2}-1|} \Big[m^{\frac{p}{2}}\left(V_{c}^{n}\left(M,[g]\right)\right)^{\frac{p}{m}}
+ m^{\frac{q}{m}}\left(V_{c}^{n}\left(M,[g]\right)\right)^{\frac{q}{m}}\Big].
\end{align*}
\end{Th}
Before giving the proof for this theorem we consider the following lemma.
\begin{Lemma}\label{lemdo}
Consider $\phi$ as same as Lemma \ref{lemyek}. If $\lambda_{1,p,q}^{D}$ denotes the first eigenvalue of the $\left(p,q\right)$-Laplacian equation (\ref{hasht}), then
\begin{align*}
\lambda_{1,p,q}^{D}\left(M\right) \leq \left(n+1\right)^{|1-\frac{p}{2}|}\Big[\int_{M}|d \psi |^{p} dv + \int_{M}|d \psi |^{q} dv \Big],
\end{align*}
where $\psi = \gamma \circ \phi$ and $\gamma \in G\left(n\right)$.
\end{Lemma}
\begin{proof}
From the definition of $\lambda_{1,p,q}^{D}\left(M\right)$ we see
\begin{align*}
\lambda_{1,p,q}^{D}\left(M\right) \leq \frac{\int_{M}|d\psi_{i}|^{p} dv + \int_{M}|d \psi_{i}|^{q} dv}{\int_{M}|\psi_{i}|^{p} dv},
\end{align*}
thus
\begin{align*}
\lambda_{1,p,q}^{D}\left(M\right) \leq \sum_{i=1}^{n+1}\frac{\int_{M} |d \psi_{i}|^{p} dv + \int_{M}|d \psi_{i}|^{q} dv}{\int_{M}|\psi_{i}|^{p} dv}.
\end{align*}
First, let $p\geq q \geq 2$, then
\begin{align*}
\sum_{i=1}^{n+1}|d \psi_{i}|^{p} &= \sum_{i=1}^{n+1}\left(|d \psi_{i}|^{2}\right)^{\frac{p}{2}} \leq \left(\sum_{i=1}^{n+1}|d \psi_{i}|^{2}\right)^{\frac{p}{2}} = |d\psi |^{p}.
\end{align*}
Also in the similar way for $q$ we have
\begin{align*}
\sum_{i=1}^{n+1}|d \psi_{i}|^{q} \leq |d \psi |^{q}.
\end{align*}
Since $\sum_{i=1}^{n+1}|\psi_{1}|^{2} =1$ and the map $x \rightarrow x^{\frac{p}{2}}$ is concave we get
\begin{align*}
\sum_{i=1}^{n+1}|\psi_{i}|^{p} \geq \left(n+1\right)^{1-\frac{p}{2}}\left(\sum_{i=1}^{n+1}|\psi_{i}|^{2}\right)^{\frac{p}{2}} = \left(n+1\right)^{1-\frac{p}{2}},
\end{align*}
and
\begin{align*}
\lambda_{1,p,q}^{D}\left(M\right) \leq \left(n+1\right)^{\frac{p}{2}-1}\Big[\int_{M}|d \psi |^{p} dv + \int_{M}|d\psi |^{q} dv \Big].
\end{align*}
Now let $1<q\leq p<2$, since $|\psi_{i}|<1$ and also $|\psi_{i}|^{2} \leq |\psi_{i}|^{p}$ and
\begin{align*}
1 = {\rm vol}\left(M,g\right) &= \int_{M} \sum_{i=1}^{n+1}|\psi_{i}|^{2} dv
\leq \int_{M} \sum_{i=1}^{n+1}|\psi_{i}|^{p} dv,
\end{align*}
we conclude that
\begin{align*}
\sum_{i=1}^{n+1}|d \psi_{i}|^{p} = \sum_{i=1}^{n+1}\left(|d \psi_{i}|^{2}\right)^{\frac{p}{2}} &\leq \left(n+1\right)^{1-\frac{p}{2}}\left(\sum_{i=1}^{n+1}|d \psi_{i}|^{2}\right)^{\frac{p}{2}}
=\left(n+1\right)^{1-\frac{p}{2}}|d \psi |^{p},
\end{align*}
similarly,
\begin{align*}
\sum_{i=1}^{n+1}|d \psi_{i}|^{q} = \sum_{i=1}^{n+1}\left(|d \psi_{i}|^{2}\right)^{\frac{q}{2}} &\leq \left(n+1\right)^{1-\frac{p}{2}}\left(\sum_{i=1}^{n+1}|d \psi_{i}|^{2}\right)^{\frac{q}{2}}
=\left(n+1\right)^{1-\frac{p}{2}}|d\psi |^{q},
\end{align*}
which finally give us
\begin{align*}
\lambda_{1,p,q}^{D}\left(M\right) \leq \left(n+1\right)^{|1-\frac{p}{2}|}\Big[\int_{M}|d \psi |^{p} dv + \int_{M}|d \psi |^{q} dv \Big].
\end{align*}
\end{proof}
\begin{proof}[{\bf Proof of Theorem \ref{set}}]
Consider $\phi: \left(M,g\right) \rightarrow \left(\mathbb{S}^{m}, can\right)$ as a conformal immersion. From Lemma \ref{lemdo} there exists $\gamma \in G\left(n\right)$ such that
\begin{align*}
\lambda_{1,p,q}^{D}\left(M\right) \leq \left(n+1\right)^{|\frac{p}{2}-1|}\Big[\int_{M}|d\psi |^{p} dv + \int_{M}|d \psi |^{q} dv \Big].
\end{align*}
By H\"older's inequality we have
\begin{align*}
\int_{M}|d\psi |^{p} dv \leq \left(\int_{M}|d \psi |^{m} dv \right)^{\frac{p}{m}},
\end{align*}
also the similar context holds for $q$. Since $\gamma \circ \phi: \left(M,g\right) \rightarrow \left(\mathbb{S}^{m}, can\right)$ is a conformal immersion and
\begin{align*}
\left(\gamma \circ \phi \right)^{*}can = \frac{|d\left(\gamma \circ \phi \right)|^{2}}{m}g,
\end{align*}
we have
\begin{align*}
\int_{M}|d \left(\gamma \circ \phi \right)|^{p} dv &= m^{\frac{p}{2}}{\rm vol}\left(M, \left(\gamma \circ \phi \right)^{*}can\right) \\
&\leq m^{\frac{p}{2}}\sup_{\gamma \in G\left(n\right)}{\rm vol}\left(M, \left(\gamma \circ \phi \right)^{*}can\right),
\end{align*}
also the similar context holds for $q$. Therefore, these together and by taking $inf$ respect to $\phi$ we find that
\begin{align*}
\lambda_{1,p,q}^{D}\left(M\right) \leq \left(n+1\right)^{|\frac{p}{2}-1|} \Big[m^{\frac{p}{2}}\left(V_{c}^{n}\left(M,[g]\right)\right)^{\frac{p}{m}}
+ m^{\frac{q}{m}}\left(V_{c}^{n}\left(M,[g]\right)\right)^{\frac{q}{m}}\Big].
\end{align*}
\end{proof}
\begin{Rem}
It seems clear that under consideration $p \geq q$, the $\left(p,q\right)$-Laplacian equation (\ref{hasht}) turns into the known $p$-Laplacian system (\ref{yek}) which was studied extensively in \cite{ma2}. So by the similar way of Matei \cite{ma2} and Theorem \ref{dot}, for an $m$-dimensional compact manifold $M$ and $p\geq q>m$ we get
\begin{align*}
\limsup_{\epsilon \rightarrow 0}\lambda_{1,p,q}^{D}\left(\epsilon \right). \epsilon^{\frac{p}{m}} = \infty.
\end{align*}
\end{Rem}

\end{document}